\documentclass[oneside]{amsart}

\usepackage{subfiles}	
\usepackage[utf8]{inputenc}
\usepackage[english]{babel}
\usepackage{amsmath,amsthm}
\usepackage{amsfonts}
\usepackage{amssymb}
\usepackage{graphicx}
\usepackage[]{geometry}
\usepackage{tikz}
\usepackage{tikz-3dplot,pgfplots}
\usepackage{color}
\usepackage{enumerate}
\usepackage[all]{xy}
\usepackage[normalem]{ulem}
\usepackage{todonotes}

\theoremstyle{definition}
\newtheorem{definition}{Definition}[section]
\newtheorem{remark}[definition]{Remark} 
\newtheorem{defn}[definition]{Definition}

\newtheorem*{claim}{Claim}

\usepackage{cleveref}

\theoremstyle{plain}
\newtheorem{theorem}[definition]{Theorem}

\newtheorem{corollary}[definition]{Corollary}
\newtheorem{lemma}[definition]{Lemma}
\newtheorem{proposition}[definition]{Proposition}  
\newtheorem{conjecture}[definition]{Conjecture}

\newtheorem{innercustomthm}{Theorem} 
\newenvironment{theorem*}[1]
{\renewcommand\theinnercustomthm{#1}\innercustomthm}
{\endinnercustomthm}

\newtheorem{innercustomprop}{Proposition} 
\newenvironment{prop*}[1]
{\renewcommand\theinnercustomprop{#1}\innercustomprop}
{\endinnercustomprop}

\newtheorem{innercustomcor}{Corollary} 
\newenvironment{corollary*}[1]
{\renewcommand\theinnercustomcor{#1}\innercustomcor}
{\endinnercustomthm}

\newcommand{\Z}{\ensuremath{{\mathbb{Z}}}}
\newcommand{\R}{\ensuremath{{\mathbb{R}}}}

\newcommand{\G}{\ensuremath{{\Gamma}}}
\newcommand{\AG}{A_{\Gamma}}
\newcommand{\WG}{\ensuremath{{W_{\Gamma}}}}
\newcommand{\DG}{\ensuremath{{\mathcal{D}_\Gamma}}} 
\newcommand{\DGP}{\ensuremath{{\mathcal{D}_\Gamma^+}}}

\newcommand{\sDGP}{\ensuremath{{\mathcal{D}_\Gamma^{\Delta +}}}}
\newcommand{\sDG}{\ensuremath{{\mathcal{D}^{\Delta}_\Gamma}}} 

\newcommand{\Cay}[2]{{\rm Cay}({#1},{#2})}

\newcommand{\MCay}{\Cay{\AG}{M}}
\newcommand{\XCay}{\Cay{\AG}{X}}

\tikzset{vertex/.style={circle, draw, fill=black!50},inner sep=0pt, minimum width=4pt}

\title{The Artin monoid cayley graph}

\author{Rachael Boyd}
\address{School of Mathematics and Statistics, University of Glasgow, Glasgow G12 8QQ, UK}
\email{rachael.boyd@glasgow.ac.uk}

\author{Ruth Charney}
\address{Department of Mathematics, Brandeis University, 415 South St., Waltham, MA 02421, USA}
\email{charney@brandeis.edu}

\author{Rose Morris-Wright}
\address{Department of Mathematics, Middlebury College, 14 Old Chapel rd, Middlebury VT, 05753, USA}
\email{rmorriswright@middlebury.edu}

\author{Sarah Rees}
\address{School of Mathematics, Statistics and Physics, University of Newcastle, Newcastle NE1 7RU, UK}
\email{Sarah.Rees@newcastle.ac.uk}

\begin{document}


\begin{abstract}

In this paper we investigate properties of the Artin monoid Cayley graph.   This is the Cayley graph of an Artin group $A_\G$ with respect to the (infinite) generating set given by the associated Artin monoid $A^+_\G$.  In a previous paper, the first three authors introduced a monoid Deligne complex and showed that this complex is contractible for all Artin groups.  In this paper, we show that the Artin monoid Cayley graph is quasi-isometric to a modification of the Deligne complex for $A_\G$ obtained by coning off translates of the monoid Deligne complex.  We then address the question of when the monoid Cayley graph has infinite diameter.  We conjecture that this holds for all Artin groups of infinite type.  We give a set of criteria that imply infinite diameter, and using existing solutions to the word problem for large-type Artin groups and 3-free Artin groups, we prove that the conjecture holds for any Artin group containing a 3-generator subgroup of one of these two types.

\end{abstract}

\maketitle

\section{Introduction} \label{Section:Introduction}

Artin groups (also known as Artin-Tits groups) form a large class of groups closely associated to Coxeter groups.  They arise naturally in algebraic geometry, topology, and representation theory.   The classical examples of such groups are braid groups, whose associated Coxeter groups are the symmetric groups.  Braid groups have been extensively studied using both combinatorial and geometric methods.  While many of these techniques can be generalized to all spherical type Artin groups (that is, those whose associated Coxeter group is finite),  Artin groups associated to infinite Coxeter groups present a much bigger challenge and remain largely mysterious.  Recently, there has been a major push to find new geometric approaches to study these groups.  In this paper, we investigate some new combinatorial approaches.  

To specify an Artin group, we begin with a labelled graph $\G$
with a finite vertex set $X=\{x_1, \dots x_n\}$ and edge set $E$, such that each edge $\{x_i,x_j\}\in E$ is labelled by an integer $m_{i,j}=m_{j,i}\geq 2$.  We then define the \textit{Artin group}, $\AG$, to be the group with presentation 
\[
\AG =\langle X \mid \underbrace{x_ix_jx_i\ldots}_{\text{length }m_{i,j}}=\underbrace{x_jx_ix_j\ldots}_{\text{length }m_{i,j} }\,\, \forall \{x_i,x_j\} \in E \rangle.
\] 
If there is no edge in $\G$ between $x_i$ and $x_j$ in~$X$, we set $m_{i,j}=\infty$ and there is no relation between $x_i$ and $x_j$ in the presentation. 
The associated Coxeter group $\WG$ is the group whose presentation is the same as $\AG$ with the added relations $x_i^2=e$ for all $i$, where $e$ is the identity element.  

If we allow only positive powers of the generators, then the presentation above gives rise to a monoid, $\AG^+$. As it turns out, questions that seem intractable for Artin groups are often easier to solve in the monoid case. For example, because all of the relations in the presentation preserve the length of a word, there is an easily computable length function for any given monoid element. Thus the word problem in the monoid is easily solved. Moreover, there are nice normal forms for elements of the monoid \cite{BrieskornSaito, Michel}. However, solving the word problem in Artin groups is much more difficult, {and for many Artin groups, it is not  even known whether such solutions exist.} 

By a theorem of Paris \cite{Paris02}, the natural map from $\AG^+$ to the submonoid of $\AG$ consisting of positive words in $X$ is an isomorphism.  So it is natural to ask if and how one might use information about the monoid to understand the Artin group.  In the case of a spherical type Artin group, a connection between the group and the monoid is obtained via the existence of a Garside element.   Namely there is an element $\Delta$ in the Artin monoid (corresponding to the longest element in the Coxeter group) such that for any $g \in \AG$, $g\Delta^n$ is represented by a positive word in $\AG^+$ for sufficiently large $n$.  In addition, $\Delta^2$ lies in the center of $\AG$.  It follows that most questions about the behavior of the Artin group can be translated into questions about the Artin monoid.  The existence of this Garside element is key to our understanding of spherical type Artin groups. (We refer the reader to \cite{DehornoyDigneGodelleKrammerMichel} and \cite{McCammond} for more details about Garside groups and Garside elements.) 

For infinite type Artin groups, there is no Garside element and no clear way to translate information from the Artin monoid to the Artin group. In this paper, we begin to address this problem by investigating the Cayley graph $\MCay$ of the Artin group $\AG$ with respect to the (infinite) generating set $M:=\AG^+$.  
{We call this the \emph{monoid Cayley graph} for $\AG$.  Geodesics in this Cayley graph correspond to minimal length multi-fractions,  $g=a_1a_2^{-1} \dots a_k^{\pm 1}$, $a_i \in \AG^+$, representing an element $g \in\AG$.  Understanding these geodesics is a first step towards defining normal forms for elements of $\AG$ and offers inductive approaches to answer some questions about these groups.}
In the case of a spherical type Artin group, $\MCay$ is neither interesting nor useful as the Garside structure implies that this graph has diameter two; we can reach any element by traveling along an edge labelled by a single monoid element followed by an edge labelled (forward or backward) by a power of $\Delta$.  
We conjecture that in all other cases, that is for all infinite type Artin groups, $\MCay$ has infinite diameter.  It is this question that we address in this paper.  Surprisingly, it is not at all obvious.  

The first step is to observe that if $A_T$ is a special subgroup of $\AG$ (that is, the subgroup generated by a subset $T \subset X$), then the inclusion of the monoid Cayley graph of $A_T$ into the monoid Cayley graph of $\AG$ is an isometric embedding (see Proposition \ref{Prop:retraction}).  Hence to show that $\MCay$ has infinite diameter it suffices to find a special subgroup for which this holds.  Thus, while we focus primarily on 3-generator Artin groups, our results extend to any Artin group containing one of these groups as a special subgroup.  Our main result is the following.

\begin{theorem*}{\ref{Thm:Main}} Let $\MCay$ denote the Cayley graph of $\AG$ with respect to the generating set $M:=\AG^+$.  Then $\MCay$ has infinite diameter providing one of the following holds.
	\begin{enumerate}
		\item $\G$ contains a pair of vertices not connected by an edge (or equivalently, there exists $i,j$ with $m_{i,j}=\infty$).
		\item $\G$ contains a triangle with all labels $\geq 3$.
		\item $\G$ contains a triangle with no label equal to $3$ and at most one label equal to $2$. 
	\end{enumerate}
\end{theorem*}

Note that the only 3-generator Artin groups of infinite type not covered by this theorem are those with edge labels $\{2,3,n\}$ with $n \geq 6$.   
{After posting this paper, we were informed by Arye Juhasz that he has a proof of the $\{2,3,n\}$ case.  Together with our results, this leaves open only the case of Artin groups all of whose 3-generator special subgroups are spherical type.}
The proof of the theorem involves constructing a sequence of geodesics of increasing length in $\MCay$.  To do this, we first establish some explicit criteria on words in the generators $X$, that allow us to construct such geodesics.  Using solutions to the word problem established  by Holt and the fourth author in the case of large type Artin groups \cite{HoltRees2012} and by Blasco-Garc\'ia, Cumplido and the third author in the case of 3-free Artin groups \cite{BlascoCumplidoMW2022}, 
we then show that these criteria hold for any 3-generator Artin group satisfying one of the conditions above.

While the proof of the main theorem involves combinatorial arguments, the original motivation came from geometric questions. A geometric construction that has played a key role in the study of Artin groups is the Deligne complex $\DG$, which we describe in Section \ref{Section:quasiisometry}.  A central conjecture about Artin groups, known as the $K(\pi,1)$-conjecture, can be reduced to proving that the Deligne complex is contractible.  
In 1972,  Deligne \cite{Deligne72} proved that  this conjecture holds for spherical type Artin groups using the combinatorial information given by the Garside structure. 
The conjecture has subsequently been proved using more geometric arguments for some special classes of infinite type Artin groups (see the survey paper by Paris \cite{Paris14}) but the general conjecture has remained unapproachable.  In a paper by the first three authors, we introduce an analogue of the Deligne complex for the monoid, called the monoid Deligne complex $\DGP$, and prove that this complex is contractible for \emph{all} Artin monoids.  The monoid Deligne complex embeds in the full Deligne complex and its translates cover all of $\DG$.  Our attempt to better understand this covering, and how these translates intersect, led us to the study of the monoid Cayley graph. 

{There is a direct connection between the geometry of the monoid Cayley graph and that of the Deligne complex.} Since $\DGP$  is contractible, attaching a cone to each translate of $\DGP$ in $\DG$ does not change its topology.  In Section \ref{Section:quasiisometry}, we define the ``coned-off Deligne complex", $C\DG$ and prove,
\begin{theorem*}{\ref{Thm:QI}} For any infinite type Artin group $\AG$, $C\DG$ is quasi-isometric to $\MCay$.
\end{theorem*}

The paper is organized as follows.  In Section \ref{Section:monoidCayley}, we introduce the monoid Cayley graph $\MCay$ and prove that  the monoid Cayley graph of a special subgroup isometrically embeds into the monoid Cayley graph of $\AG$.   In Section \ref{Section:quasiisometry}, we investigate the coned-off Deligne complex.  Sections \ref{Section:criteria} and \ref{Section:largetype} contain the proof of the main theorem.  

\subsection*{Acknowledgements}
The authors would like to thank Derek Holt for helpful conversations and suggestions, and the anonymous referee for helpful comments.  The first author was supported by EPSRC Fellowship No.~EP/V043323/2.
The authors would also like to thank ICERM for hosting the 2022 semester on Braids where some discussions for this project took place.

\section{The monoid Cayley graph} \label{Section:monoidCayley}
In this section we introduce some basic notation and discuss the relationship between the monoid Cayley graph of an Artin group and the monoid Cayley graphs of special subgroups.  

A standard tool in geometric group theory for studying a finitely generated group $G$ is the assignment of a metric to $G$. To do this, one chooses a finite generating set $S$ and considers the Cayley graph of $G$ with respect to $S$, denoted $\Cay{G}{S}$.  This is the graph whose vertices are the elements of $G$ and edges join two vertices $g$ and $h$ when they differ by a single element of $S$, that is, $h=gs$ for some $s \in S$.  One can then define a metric on $G$ by taking the distance from $g_1$ to $g_2$ to be the length of a shortest path from $g_1$ to $g_2$; in this Cayley graph; each such path is labelled by a word over $S$ (string over $S \cup S^{-1}$).  Different choices of (finite) generating sets give rise to quasi-isometric metrics, so this metric is ``coarsely" well-defined.  For an Artin group $\AG$ with its standard generating set $X$, we denote this Cayley graph by $\XCay$.

In this paper we shall consider a different type of Cayley graph, namely the Cayley graph with respect to an infinite generating set $M$.  As defined above, $M$ consists of all elements in the Artin monoid $\AG^+$.  The vertices of this Cayley graph, $\MCay$, are again the elements of $\AG$, but two vertices $g,h$ are now joined by an edge whenever $h=gm$ for some $m \in M$.  Note that these edges come with an orientation from $g$ to $h$.  
However, paths in this Cayley graph may follow edges in either direction.  Thus a path from the identity vertex $e$ to a vertex $g$ corresponds to a word over $M$ (string over $M \cup M^{-1}$) representing $g$.  Given a word $w$ over the standard generating set $X$, we can subdivide $w$ into maximal subwords which are all positive or all negative.  This determines a word over $M$ representing the same element of $\AG$.  We will refer to the length of this word, viewed as a path in $\MCay$, as the \emph{monoidal length} of $w$ and denote it by $|w|_M$.

We are interested in understanding the geometry of $\MCay$.  
Here are a few easy observations.  First, if $\AG$ is a spherical type Artin group with Garside element $\Delta$, then any $g \in \AG$ can be translated into the monoid by multiplying by a sufficiently high power of $\Delta$,  that is, there exists $n$ such that $g \Delta^n =m \in M$.  Thus, there is a path from the identity vertex $e$ to $g$ consisting of two edges, the edge from $e$ to $m$, followed by the (reverse of) the edge from $g$ to $m$.  It follows that in this case, the diameter of $\MCay$ is two. Thus, the monoid Cayley graph is of little interest in the spherical type case.  We will restrict our attention to infinite type Artin groups.

\begin{conjecture}\label{Conj:InfDiam}  For any infinite type Artin group, $\MCay$ has infinite diameter.
\end{conjecture}

This conjecture will be the main focus of this paper.  Studying the diameter involves understanding geodesics in this Cayley graph. These are of interest in their own right.  Note that any two vertices connected by a path all of whose edges have the same orientation will also be connected by a single edge, since the product of monoid elements is again a monoid element.  Thus, any geodesic from $e$ to $g$ in $\MCay$ must be an alternating path, or, in other words, an expression  of the form $g=a_1 a_2^{-1} a_3 \dots a_k^{\pm 1}$ (or $g=a_1^{-1} a_2 a_3^{-1} \dots a_k^{\pm 1}$) with $a_i \in M$.  

In this paper, we shall focus on the case of 3-generator Artin groups.  However, as we will now see, this has immediate implications for higher rank Artin groups.  Let $\AG$ be an Artin group with standard generating set $X$.  A \emph{special subgroup} of $\AG$ is a subgroup  $A_T$ generated by a subset $T \subseteq X$.  By a theorem of van der Lek \cite{vdL83}, this group is itself an Artin group with defining graph the full subgraph of $\G$ spanned by $T$.  

\begin{proposition}\label{Prop:retraction} Let $\AG$ be an Artin group with generating set $X$ and let $T \subset X$ be any non-empty subset.  Let $M_T$ be the positive monoid in $A_T$.  There is a simplicial retraction of $\MCay$ onto $\Cay{A_T}{M_T}$.  Hence the natural inclusion of $\Cay{A_T}{M_T}$ into $\MCay$ is an isometric embedding, that is, it maps geodesics to geodesics.
\end{proposition}

\begin{proof}This follows from work of Godelle and Paris \cite{GP12} in which they construct a retraction of the Salvetti complex of $\AG$ to the Salvetti complex of $A_T$.  Subsequently, Charney and Paris \cite{CP14} studied the universal cover of this retraction in more detail.  The 1-skeleton of the universal cover of the Salvetti complex of $\AG$ is precisely the Cayley graph of $\AG$ with respect to the standard generating set $X$. Edges come with a preferred orientation, namely the orientation corresponding to the positive power of a generator in $X$.   In the proof of Theorem 1.2 of \cite{CP14}, they give a precise description of this retraction: the retraction either collapses an edge to a single point, or maps it to an edge with the same orientation.  It follows that it takes any strictly positive path, corresponding to a monoid element in $\AG$, to a strictly positive path corresponding to a monoid element in $A_T$.  Thus, it induces a retraction of $\MCay$ onto $\Cay{A_T}{M_T}$  taking each edge, labelled by an element of $M$, to either a single vertex or a single edge labelled by an element of $M_T$. In particular, if $\gamma$ is an edge path in $\MCay$ connecting two vertices of $\Cay{A_T}{M_T}$, then its image under the retraction is an edge path of the same or shorter length in  $\Cay{A_T}{M_T}$ between these vertices.  It follows that geodesic paths in $\Cay{A_T}{M_T}$ must also be geodesic in $\MCay$.
\end{proof}

\begin{corollary}\label{Cor:subgroups} If $\AG$ contains a special subgroup $A_T$ whose monoid Cayley graph has infinite diameter, then the same holds for the monoid Cayley graph of $\AG$.  
\end{corollary}

As an example, let $T=\{s,t\}$ be two generators not connected by an edge, so $A_T$ is the free group on $\{s,t\}$.   It is easy to see that the word $st^{-1}st^{-1} \dots st^{-1}$ of length $2m$ is geodesic in $\Cay{A_T}{M_T}$  .  It then follows from the corollary that for any $\G$ whose underlying graph is not a clique $\MCay$ has infinite diameter.

\section{Relationship to the monoid Deligne complex} \label{Section:quasiisometry}
In this section we introduce the coned off Deligne complex (which we note here is different from the coned off Deligne complex defined in the paper of Martin and Przytycki \cite{MartinPrzytycki}) and show that the resulting complex is quasi-isometric to $\MCay$.

We first give a bit of history.  Associated to any Coxeter group $W_\G$ is a complex hyperplane complement on which $W_\G$ acts freely.  It is known that the quotient of this space by $W_\G$ has fundamental group $\AG$.  The K($\pi$,1)-conjecture states that, in fact, this quotient space is a K($\AG$,1)-space.  In 1972, Deligne \cite{Deligne72} proved that this holds for spherical type Artin groups, by constructing a simplicial complex homotopy equivalent to the universal cover of this hyperplane complement and proving that this simplicial complex is contractible.  In \cite{CD95}, the second author and M.~Davis constructed a slight variation on this complex for infinite-type Artin groups, now known as the Deligne complex, which plays an analogous role for infinite-type Artin groups.  The K($\pi$,1)-conjecture is equivalent to the conjecture that the Deligne complex is contractible for all $\AG$. In \cite{CD95} the conjecture was proved in some special cases (such as for FC type) and more recently, it was proved for affine type  Artin groups by Paolini and Salvetti \cite{PS21}, but it remains open in general.

To define the Deligne complex, first note that a special subgroup $A_T$ of an infinite type Artin group $\AG$ may have either spherical type or infinite type.  The Deligne complex $\DG$  is the geometric realization of the partially ordered set $\mathcal P$ of cosets $gA_T$, where $A_T$ is a spherical type special subgroup and the ordering is given by inclusion.  {(Note that we allow $T=\emptyset$ and define $gA_\emptyset=\{g\}$.)} There are two natural piecewise Euclidean metrics that we can put on $\DG$.  The first is obtained by using the simplicial structure and declaring each simplex to be a standard Euclidean simplex with all edges of length 1.  We denote the Deligne complex with this metric by $\sDG$.  An alternate metric is obtained by viewing $\DG$ as a cube complex.  Here, the cubes correspond to intervals $[gA_T, gA_R]$ in the poset $\mathcal{P}$ and the metric assigns each k-cube to be isometric to the unit cube $I^k$ in $\R^k$.  These two metrics on $\DG$ are quasi-isometric since the maximal dimension of cubes in the Deligne complex is finite, so viewing the cube as a union of simplicies distorts the metric by a bounded amount.  We will use the simplicial structure and metric defined on $\sDG$ as it is more convenient for our arguments.  

We can also define an analogue of this space for the Artin monoid.  Namely, we define $\sDGP$ to be the geometric realization of the subposet of $\mathcal P$ consisting of cosets $mA_T$ where $m \in \AG^+=M$.  In \cite{BCMW22}, the first three authors prove that $\sDGP$ is contractible for all infinite type $\AG$. 
The complex  $\sDGP$ naturally embeds into $\sDG$. 
There is an action of~$\AG$ on $\sDG$, via left multiplication on the cosets~$g\cdot m A_T=gmA_T$.
We call the image of~$\sDGP$ under the action of~$g\in \AG$ a \emph{translate} of~$\sDGP$ and denote it by~$g\sDGP$.

\begin{defn}
	Define the \emph{coned off Deligne complex} $C\DG$ to be the simplicial Deligne complex~$\sDG$ with a `cone point' added for each translate of~$\sDGP$, in the following sense: for all~$g \in \AG$, add
	\begin{enumerate}
		\item a vertex $v_g$ 
		\item an edge between every vertex in~$g\sDGP$ and the vertex~$v_g$
		\item a~$p+1$ simplex~$v_g*\sigma_p$ for every~$p$-simplex~$\sigma_p\in g\sDGP$.
	\end{enumerate}
\end{defn}

This is again a simplicial complex, and we put the standard metric on each simplex, such that all edges have length 1.

One motivation for defining this coned off complex, is that we know that the monoid Deligne complex is contractible, as are its translates. Therefore coning them off does not change the homotopy type. \emph{i.e.}:
\[
C\DG\simeq \sDG.
\]

We now define a map from vertices in~$\MCay$ to vertices in~$C\DG$ as follows:
\[f:\MCay \to C\DG; \, g \mapsto v_g \] 

\begin{theorem}\label{Thm:QI}
	The above map~$f$ is a quasi-isometry, {i.e.}~$\MCay\simeq_{q.i.}C\DG$.
\end{theorem}

\begin{proof}
	First, note that because the translates of~$\sDGP$ by~$\AG$ cover~$\sDG$, any point in~$\sDG$ is in some translate~$g\sDGP$, and therefore at most distance one from the cone point for that translate, $v_g$, in~$C\DG$. Thus every point in~$C\DG$ is at most distance one from a point in~$f(\MCay)$.
	
	First, we claim that if~$g$ and~$h$ are joined by an edge in~$\MCay$ then~$f(g)=v_g$ and~$f(h)=v_h$ are joined by an edge path of length 2 in~$C\DG$. Without loss of generality we can assume the edge is directed from~$h$ to~$g$, so that~$g=h\cdot m$ for some~$m\in M$. Now~$f(g)=v_g$ is joined by an edge to~$gA_\emptyset$, and $gA_\emptyset=h\cdot m A_\emptyset$ is in~$h\sDGP$, so it is joined by an edge to~$v_h=f(h)$. We have therefore constructed a path of length 2 from~$f(g)$ to $f(h)$ as claimed.

	Since any geodesic in~$\MCay$ is a path of such edges, for~$x$ and~$y$ in~$\MCay$ we deduce that~
	\[d_{C\DG}(f(x),f(y))\leq 2\cdot d_{Cay}(x,y).\] 
	In particular this means that for any~$A\in\R$ such that~$A\geq 2$ we have~ $d_{C\DG}(f(x),f(y))\leq A\cdot d_{Cay}(x,y)$. It remains to show that~$d_{Cay}(x,y)\leq A \cdot d_{C\DG}(f(x),f(y))$ for some $A$.
	
	\begin{figure}[h!]\label{figure:quasiiso}
		\begin{center}
		\begin{tikzpicture}[scale=1.2]
			\foreach \x in {0,1,2,3,5,6,7}
			\draw[fill=black] (\x,0) circle [radius=0.1];
			\draw[thick] (0,0) -- (3.3,0) (4.7,0)--(7,0);
			\foreach \x in {3.7,4,4.3}
			\draw[fill=black] (\x, 0) circle [radius=0.02];
			
			\draw[->] (1.3,1) --(1,0.2);
			\draw (1.3,1.2) node {$\scriptstyle{xm_1A_T=g_2m_2A_T}$};
			\draw[->] (6.3,1) --(6,0.2);
			\draw (6.3,1.2) node {$\scriptstyle{g_{k-1}m_{k-1}A_T=ym_kA_T}$};

			\draw (-2,0) node {$\scriptstyle{\text{Edge path }\gamma' \in C\DG}$};
			\draw (-0.2,0.2) node{$\scriptstyle{v_x}$} (7.2,0.15) node{$\scriptstyle{v_y}$}(0.5,0.15) node{$\scriptstyle{e_1}$}(1.5,0.15) node{$\scriptstyle{e_2}$} (2.5,0.15) node{$\scriptstyle{e_3}$}(5.5,0.15) node{$\scriptstyle{e_{k-1}}$}(6.5,0.15) node{$\scriptstyle{e_k}$};
			
			\draw[thin,gray] (0,-2) -- (3.3,-2) (4.7,-2)--(7,-2);
			\foreach \x in {0,1,2,3,5,6,7}
			\draw[fill=black] (\x,-2) circle [radius=0.1];
			\foreach \x in {3.7,4,4.3}
			\draw[gray, fill=gray] (\x, -2) circle [radius=0.02];
			
			\foreach \x in {1.5,2.5,5.5}
			\draw[fill=black] (\x,-3) circle [radius=0.1];
			\draw[thick] (0,-2) -- (1,-2) --(1.5,-3)--(2,-2)--(2.5,-3)--(3,-2)--(3.2,-2.4) (4.8,-2.4)--(5,-2)--(5.5,-3)--(6,-2)--(7,-2);
			\foreach \x in {3.7,4,4.3}
			\draw[fill=black] (\x, -2.6) circle [radius=0.02];
			
			\draw (-2,-2) node {$\scriptstyle{\text{Modified path } \in C\DG}$};
			\draw (-0.2,-1.8) node{$\scriptstyle{v_x}$} (7.2,-1.85) node{$\scriptstyle{v_y}$}(0.5,-1.85) node{$\scriptstyle{e_1}$}(1.5,-1.85) node[gray]{$\scriptstyle{e_2}$} (2.5,-1.85) node[gray]{$\scriptstyle{e_3}$}(5.5,-1.85) node[gray]{$\scriptstyle{e_{k-1}}$}(6.5,-1.85) node{$\scriptstyle{e_k}$};
			
			\draw (1.5,-3.3) node{$\scriptstyle{v_{g_2}}$} (2.5,-3.3) node{$\scriptstyle{v_{g_3}}$} (5.5,-3.3) node{$\scriptstyle{v_{g_{k-1}}}$} (1,-2.5) node{$\scriptstyle{d^1_2}$} (1.9,-2.6) node{$\scriptstyle{d^2_2}$} (5,-2.7) node{$\scriptstyle{d^1_{k-1}}$} (6.1,-2.7) node{$\scriptstyle{d^2_{k-1}}$};

			\foreach \x in {0,1,3,4,6,7}
			\draw[fill=black] (\x,-5) circle [radius=0.1];
			
			\foreach \x in {0.5,3.5,6.5}
			\draw[fill=black] (\x,-6) circle [radius=0.1];
			\draw[thick] (0,-5)-- (0.5,-6)-- (1,-5) --(1.2,-5.4)(2.8,-5.4)--(3,-5)--(3.5,-6)--(4,-5)--(4.2,-5.4) (5.8,-5.4)--(6,-5)--(6.5,-6)--(7,-5);
			\foreach \x in {1.7,2,2.3,4.7,5,5.3}
			\draw[fill=black] (\x, -5.6) circle [radius=0.02];
			\draw[thick, ->] (0,-5)--(0.25,-5.5) ;
			\draw[thick, ->](1,-5)--(0.75,-5.5) ;
			\draw[thick, ->] (3,-5)--(3.25,-5.5) ;
			\draw[thick, ->](4,-5)--(3.75,-5.5) ;
			\draw[thick, ->] (6,-5)--(6.25,-5.5) ;
			\draw[thick, ->](7,-5)--(6.75,-5.5) ;
			
			\draw (-2,-5) node[align=center] {$\scriptstyle{\text{Corresponding path }}$\\$\scriptstyle{ \in \MCay}$};
			\draw (-0.1,-4.8) node{$\scriptstyle{x}$} (7.1,-4.8) node{$\scriptstyle{y}$}(1,-4.8) node{$\scriptstyle{g_2}$} (3,-4.8) node{$\scriptstyle{g_i}$}(4,-4.8) node{$\scriptstyle{g_{i+1}}$}(6,-4.8) node{$\scriptstyle{g_{k-1}}$};
			
			\draw (0.5,-6.5) node[align=center]{$\scriptstyle{xm_1a_1}$\\$\scriptstyle{=g_2m_2b_1}$} (3.5,-6.5) node[align=center]{$\scriptstyle{g_im_ia_i}$\\$\scriptstyle{=g_{i+1}m_{i+1}b_i}$} (6.5,-6.5) node[align=center]{$\scriptstyle{g_{k-1}m_{k-1}a_k}$\\$\scriptstyle{=ym_kb_k}$} (0,-5.6) node{$\scriptstyle{m_1a_1}$} (1,-5.6) node{$\scriptstyle{m_2b_1}$} (3,-5.6) node{$\scriptstyle{m_ia_i}$} (4.1,-5.7) node{$\scriptstyle{m_{i+1}b_{i}}$} (5.9,-5.7) node{$\scriptstyle{m_{k-1}a_k}$} (7.1,-5.7) node{$\scriptstyle{m_kb_k}$};
		\end{tikzpicture}
	\end{center}
		\caption{The edge path~$\gamma'$ in $C\DG$ joining $v_x$ to $v_y$, the corresponding modified edge path in $C\DG$ (with $\gamma'$ shown in grey) and the path in~$\MCay$ from $x$ to~$y$.}
	\end{figure}
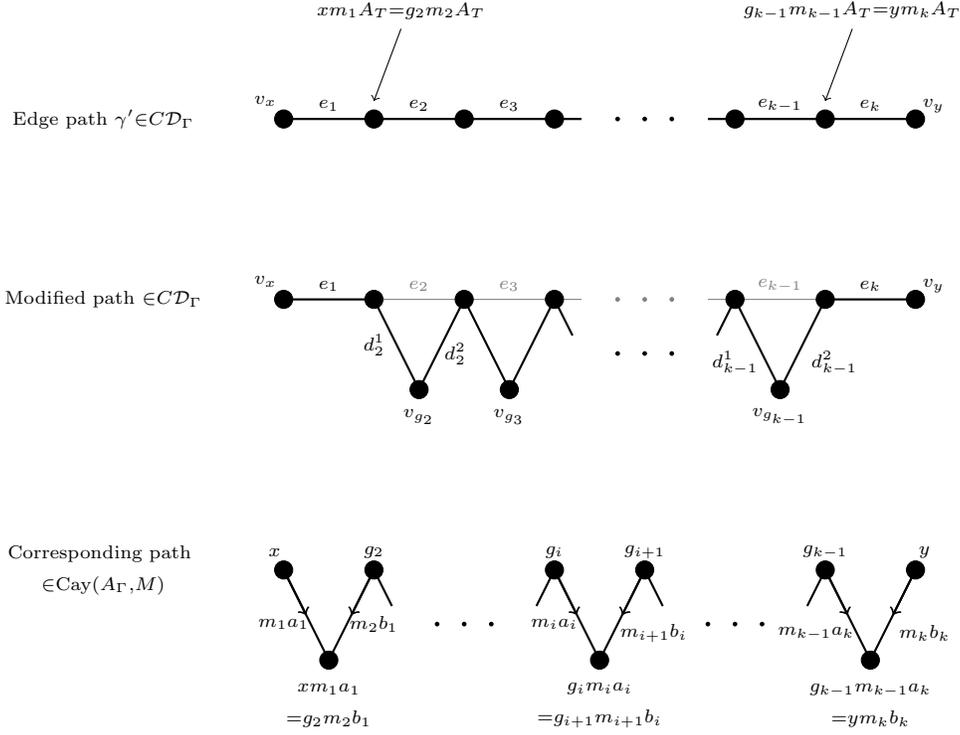
	
	For the following argument, refer to Figure 1. 
	Start with a geodesic~$\gamma$ between~$f(x)$ and~$f(y)$ in~$C\DG$. {We can deform $\gamma$ into an edge path~$\gamma'$ from $x$ to $y$ contained in the same simplices as $\gamma$.}  Since the maximal dimension of the simplices in~$\sDG$ is finite, there exists some constant~$B\geq1$, {depending only on this dimension}, such that~$B\lg(\gamma)\geq \lg(\gamma')$, where $\lg(\,)$ denotes the length of the path in~$C\DG$.  Label the edges of~$\gamma'$ by~$e_1,\ldots, e_k$.
	Each edge~$e_i$ of~$\gamma'$ is either an edge to or from a cone point, or an edge in~$\sDG$.  Since the translates of~$\sDGP$ cover~$\sDG$, any edge in~$\sDG$ lies in some translate of~$\sDGP$. Without loss of generality, we may assume that the only edges with cone point vertices are the first and last (with vertices~$v_x$ and $v_y$), as otherwise we can apply this proof to every section of the edge path between cone point vertices. Then, except for~$e_1$ and~$e_k$, every edge lies in a translate of~$\sDGP$ by an element~$g_i\in \AG$, for~$2\leq i\leq k-1$. We replace the edge~$e_i$ with the two edges $d^1_i$, $d^2_i$, with~$d^1_i$ travelling from the initial vertex of~$e_i$ to the cone point~$v_{g_i}$, and~$d^2_i$ from the cone point to the terminal vertex of~$e_i$. Edges~$e_i$ and~$e_{i+1}$ meet at a vertex in the overlap of the translates~$g_i\sDGP$ and~$g_{i+1}\sDGP$, \emph{i.e.}~the vertex satisfies~$g_im_iA_T=g_{i+1}m_{i+1}A_T$ for~$m_i$ and~$m_{i+1}$ in~$M$. Since~$A_T$ is of spherical type, this means in particular that there exist~$a_i, b_i\in A_T^+$ such that
	\[
	g_im_ia_ib_i^{-1}=g_{i+1}m_{i+1} \implies g_im_ia_i=g_{i+1}m_{i+1}b_i.
	\]
	In~$\MCay$, the vertices~$g_i$ and~$g_{i+1}$ are therefore joined by an edge path of length at most 2. 
	Similarly, the vertices~$x$ and $g_2$ are joined by an edge path of length 2, as are the vertices~$g_{k-1}$ and~$y$. This is because, for example,~$v_x$ is joined by the edge $e_1$ to a vertex which also lies in the edge~$e_2$, and thus in~$g_2\sDGP$. So this vertex satisfies~$xm_1A_T=g_{2}m_{2}A_T$ for~$m_1$ and~$m_{2}$ in~$M$ and as before we can argue that this corresponds to an edge path of length 2.
	From the modified path in~$C\DG$, we are therefore able to construct a path between~$x$ and~$y$ in~$\MCay$, as shown in Figure \ref{figure:quasiiso}.	{It follows that 
		\[
		d_{\MCay}(x,y)\leq 2k-2 \leq 2k = 2\lg(\gamma'),
		\]
		and hence 
		\[
		d_{\MCay}(x,y)\leq 2B\lg(\gamma)=2B \cdot d_{C\DG}(f(x),f(y)).
		\]}
	Taking~$A=2B\geq 2$ this completes the proof.
\end{proof}

\section{criteria for infinite diameter} \label{Section:criteria}

	In this section, we will introduce two criteria which an Artin group may or may not satisfy, and prove that if an Artin group $\AG$ satisfies both of these criteria then as a consequence $\MCay$ has infinite diameter. 
	To illustrate the criteria, we will include some simple examples of Artin groups which satisfy them  
	and Artin groups which do not. 
	Further examples appear in \Cref{Section:largetype}, where we prove that large type and 3-free Artin groups of infinite type satisfy the two criteria.

	Assuming an Artin group $\AG$ satisfies the two criteria, our strategy to prove that $\MCay$ has infinite diameter is as follows. We  start with a particular sequence of words $w_1,w_2,\dots ,w_n,\dots$ over the standard generating set~$X$, for which $|w_n|_M$, is equal to $n$. This sequence exists due to Criterion 2.  Using Criterion 1, we now show that when we premultiply $w_n$ by a word of monoidal length less than $n$ the word we obtain cannot reduce to the empty word, and hence cannot represent the identity. It follows that~$w_n$ is geodesic in~$\MCay$, and so the sequence of words $\{w_n\}$ represents a sequence of geodesics of increasing length in $\MCay$. This shows that~$\MCay$ has infinite diameter.
	
	\subsection{Criterion 1}
	
	\textbf{$\AG$ has preserved signed suffixes} 
	
	\begin{definition}
		Suppose that $\AG$ is an Artin group with generating set $X$. 
		Let $w$ be a geodesic word {over $X$} that ends in a positive (resp.~negative) letter and let $u$ denote the longest positive (resp.~negative) proper suffix of $w$. Let $a$ be a single positive (resp.~negative) generator. We say that $\AG$ has \textit{preserved positive (resp.~negative) suffixes} if we can guarantee the existence of a word for the element of the group given by $aw$ {that is geodesic in X} and also has $u$ as a suffix. 
		If $\AG$ has both preserved positive and negative suffixes then we say the group has \textit{preserved signed suffixes}.
	\end{definition}  
	
	For example, free groups have preserved signed suffixes. In a free group, if $w$ is a geodesic with respect to the standard generators (\emph{i.e.}~a freely reduced word), and $aw$ is not, this means that  $w$ must begin with the letter $a^{-1}$. A geodesic representative for the group element given by $aw$ can thus be obtained via free cancellation of $aa^{-1}$ and this leaves the remainder of $w$ (in particular the longest positive suffix~$u$) unchanged. 
	
	Similarly, the free abelian group $\Z^n$, has preserved signed suffixes. If $w$ is a geodesic word with respect to the standard generators and $aw$ is not, this means that $w$ can be factored as $w=w_1a^{-1}w_2$, and a geodesic word for the group element given by $aw$ can be obtained by allowing $a$ to commute with each individual letter of $w_1$ to obtain the word $w_1aa^{-1}w_2$. After the final cancellation of $aa^{-1}$, a geodesic is obtained and the suffix $w_2$ is unchanged. The longest positive suffix of $w$ must also be a suffix of $w_2$ so this positive suffix is preserved. A symmetric argument shows that $\Z^n$ also has preserved negative suffixes and so has preserved signed suffixes. 
	
	In our results, we use the solution to the word problem for large type Artin groups given in \cite{HoltRees2012} and for 3-free Artin groups given in \cite{BlascoCumplidoMW2022}. In \Cref{Section:largetype}, we explain in more detail how these solutions to the word problem guarantee that large type and 3-free Artin groups have preserved signed suffixes.

	\subsection{Criterion 2}
	
	\textbf{$\AG$ contains an alternating blocking sequence $\{\alpha_n\}$}
	
	In this subsection, we will describe the criterion required to construct the 
	desired sequence of words $\{w_n\}$. 
	
	\begin{definition}
		
		Let $\AG$ be an Artin group with generating set $X$. Let $u$ be a word and $x\in X\cup X^{-1}$. We say that $(u,x)$ is a \textit{blocking pair} if given any geodesic word over~$X$ of the form $wu$, it follows that the word $wux$ is also a geodesic word over $X$. 
		
	\end{definition}
	
	Intuitively, the subword $u$ is blocking any letters from $w$ from canceling out the letter $x$.  
	
	For example, in a free group if $a,x \in X$ (both positive) then $(a,x)$ is a blocking pair. 
	In a free abelian group, if $a,x \in X$ and $a\neq x$, then $(a,x)$ is not a blocking pair because the word $x^{-1}a$ is geodesic, 
	but the word $x^{-1}ax$ is not. 
	
	We now use blocking pairs to build a sequence of words over $X$ that label arbitrarily long paths in $\MCay$. 
	
	\begin{definition}
		Suppose that $\{\alpha_n\}$ is an infinite sequence of words over $X$, where $\alpha_i$ is a nonempty positive word over $X$ if $i$ is odd and $\alpha_i$ is a nonempty negative word if $i$ is even. We say that the sequence $\{\alpha_n\}$ is an \textit{alternating blocking sequence} if for any letter $x$ in $\alpha_n$ there exists some subword $u$ of $\alpha_{n-1}\alpha_n$ immediately preceeding $x$ such that $(u,x)$ is a blocking pair.
	\end{definition}
	
	Our second criterion states that~$\AG$ must contain such an alternating blocking sequence.

	\begin{remark}
		If $\AG$ is of spherical type, then $\AG$ has no alternating blocking sequence. For suppose that $u$ is any positive word and that $x^{-1}\in X^{-1}$. Then $(u,x^{-1})$ cannot be a blocking pair. To prove this, we first recall the following facts about spherical type Artin groups from Garside theory: the square of the Garside element $\Delta^2$ is in the center of $\AG$ and satisfies that for all~$x\in X$, $x$ is a right divisor of $\Delta$. Now consider the word $\Delta^2u$, which represents an element of the monoid and so is a geodesic word over~$X$. On the other hand $\Delta^2ux^{-1}$ represents the same group element as $u\Delta^2x^{-1}$ which is not freely reduced (since $x$ is a right divisor of~$\Delta^2$).  Therefore $(u,x)$ is not a blocking pair.  It follows that no alternating sequence of words $\{\alpha_n\}$ can satisfy the conditions for an alternating blocking sequence.
	\end{remark}
	
	\subsection{Sufficiency of these criteria} In this subsection, we prove that if an Artin group $\AG$ satisfies Criterion 1 and Criterion 2, then the monoid Cayley graph~$\MCay$ has infinite diameter.

	\begin{theorem} \label{Thrm:Criteria}
		Let $\AG$ be an Artin group with standard generating set $X$ and positive monoid $M$. Suppose that $\AG$ has preserved signed suffixes, and that $\AG$ contains an alternating blocking sequence of words $\{\alpha_n\}$. Then the monoid Cayley graph~$\MCay$ has infinite diameter.
		
	\end{theorem}

	\begin{proof}
		Let $w_n=\alpha_1\cdots \alpha_n$. We will show that the words $\{w_n\}$ label geodesic paths of length $n$ in $\MCay$. The claim that~$\MCay$ is infinite diameter follows.
		
		Consider some arbitrary word $v_m=\beta_m\dots \beta_2\beta_1$ where $\beta_i$ is a positive word over $X$ if $i$ is odd and $\beta_i$ is a negative word over $X$ if $i$ is even. 
		We allow the possibility that $\beta_1$ is the empty word, but for $i>1$ we assume that $\beta_i$ is non-empty; hence $v_m$ labels a path of length $m$ or $m-1$ in $\MCay$.
		
		To show that $w_n$  is a geodesic word over~$M$, it is sufficient to show that for all $v_m$ such that $m\leq n$, the word $v_mw_n$ does not represent the identity element in~$\AG$. Because the identity is represented by the empty word regardless of which generating set we use, we do this by showing that $v_mw_n$ regarded as a word over~$X$ cannot be reduced to the empty word. This follows from the below claim.
		
		\begin{claim}
			If $m\leq n$, then the group element represented by the word $v_mw_n$ has a geodesic representative over $X$ ending in $\alpha_m \cdots \alpha_{n-1}\alpha_n$. 
		\end{claim}

		We prove this claim via induction on $m$. 
		
		{\bf Base case:} Assume $m=1$. 
		We prove that for all~$n$, $v_1w_n$ is a geodesic word over $X$, and is therefore the representative we require.
		If $n=1$, then $v_1w_1=\beta_1\alpha_1$ is a positive word over $X$ and so must be geodesic over $X$; it is therefore a geodesic with suffix $\alpha_1$ as required. 
		
		Now we build the word~$v_1w_n$, and see that it is geodesic. We suppose that $j>1$ and that $v_1w_{j-1}$ is a geodesic word, and build the word $v_1w_j$ from $v_1w_{j-1}$ by appending one letter of~$\alpha_j=x_{j,1}x_{j,2}\dots x_{j,{k_j}}$ at a time and using induction on the number of letters we have appended.
		As the base case, we append the first letter~$x_{j,1}$. Then the fact that~$\{\alpha_n\}$ is an alternating blocking sequence guarantees the existence of some suffix $u_{j,1}$  of the word $v_1w_{j-1}$  such that $(u_{j,1},x_{j,1})$ is a blocking pair. Since we assumed that $v_1w_{j-1}$ is a geodesic word over $X$, the blocking pair condition implies that $v_1w_{j-1}x_{j,1}$ must also be a geodesic word over $X$.
		Now assume the inductive hypothesis that~$v_1w_{j-1}x_{j,1}\ldots x_{j,i-1}$ is a geodesic word over~$X$ and append the letter~$x_{j,i}$. Again, since~$\{\alpha_n\}$ is an alternating blocking sequence there exists some suffix $u_{j,i}$  of the word $v_1w_{j-1}x_{j,1}\cdots x_{j,i-1}$  such that $(u_{j,i},x_{j,i})$ is a blocking pair. Since $v_1w_{j-1}x_{j,1}\cdots x_{j,i-1}$ is a geodesic word over $X$ the blocking pair condition implies that $v_1w_{j-1}x_{j,1}\cdots x_{j,i}$ must also be a geodesic word over $X$.
		Iterating this method to append~$\alpha_2\ldots \alpha_n$ to~$v_1w_1=v_1\alpha_1$, we see at each step a letter $x_{j,i}$ is appended and the word remains geodesic over~$X$. Therefore $v_1w_n$ is a geodesic word over $X$ with suffix $\alpha_1\cdots \alpha_n$.
		
		{\bf Inductive hypothesis:} 
		Suppose that $m>1$. We assume that for all~$m\leq n$, the element of $\AG$ represented by $v_{m-1}w_n$ can be represented by a geodesic word $\xi$ over~$X$ that ends with $\alpha_{m-1} \cdots \alpha_n$.

		{\bf Inductive step:} Consider the case~$n=m$. By the inductive hypothesis the group element represented by $v_{m-1}w_m$ is represented by a geodesic word $\xi$ over~$X$ that ends with $\alpha_{m-1}\alpha_m$. 
		We assume that $m$ is odd, and note that a symmetric argument will work for the $m$ even case.
		Recall the word $v_m$ satisfies $v_m=\beta_mv_{m-1}$ and hence $v_mw_n=\beta_mv_{m-1}w_n=\beta_m\xi$.  
		Suppose $\beta_m=b_{l_m}\cdots b_2b_1$, \emph{i.e.}~{$\beta_m$} is of length~$l_m$ over $X$. 
		Since $m$ is odd, $\alpha_m$ and $\beta_m$ are both positive words over $X$, so each~$b_i\in X$. Starting with $b_1$, we append these letters to the 
		{left} of $\xi$ one at a time.
		The group $\AG$ has preserved signed suffixes, so when we append the positive letter $b_1$ to the {left} of the geodesic word $\xi$, we can obtain a new geodesic that ends in the same positive suffix as $\xi$, namely the suffix $\alpha_m$. Repeating this process for each~$b_i$ we obtain a geodesic word over~$X$ which represents the element $v_mw_m$ and has suffix $\alpha_m$.
		
		For $n>m$, we repeat the argument from the base case, in which we started with the geodesic representative of the group element given by $v_1w_{j-1}$  and built one for~$v_1w_j$. In this case, we start with the geodesic word we just constructed, which represents the element given by $v_{m}w_m$ and ends with~$\alpha_m$. As in the base case we append the letters $x_{j,i}$ one at a time on the right, for $m<j\leq n$, and due to the fact that~$\{\alpha_n\}$ is an alternating blocking sequence, the word remains a geodesic over $X$ at each step. After we have appended all the $x_{j,i}$ we obtain a geodesic representative of $v_mw_n$ ending in $\alpha_m\dots \alpha_n$ as desired. 
	\end{proof}
	
	{\begin{remark} In the above proof, we allowed the possibility that $|v_n|_M = |w_n|_M$ providing the last word $\beta_1$ of $v_n$ has the same sign as the first word $\alpha_1$ of $w_n$.  Thus, we have also shown that the element of $\AG$ represented by $w_n$ cannot be represented by a geodesic $v_n^{-1}$ in $\MCay$ beginning with a word of the opposite sign. \end{remark}}
	
\section{The large type and 3-free case} \label{Section:largetype} 

In this section we show that large type Artin groups and 3-free Artin groups that are not of spherical type satisfy the criteria in Section \ref{Section:criteria} that we need in order to apply Theorem~\ref{Thrm:Criteria}. 

An Artin group is said to have large type if all the associated parameters $m_{ij}$ are at least 3 (possibly infinite). A 3-free Artin group is an Artin group for which none of the parameters $m_{ij}$ are equal to 3.  

\begin{lemma}\label{Lem:3-free} Let $\AG$ be a 3-free Artin group  of infinite type with generating set X. Then either some $m_{ij}= \infty$, or $\G$ contains a triangle with at most one edge labelled 2.
\end{lemma}

\begin{proof}  Assume $\AG$ has no $m_{ij}= \infty$.  Since $\AG$ is infinite type, it must contain at least 3 generators, that is, $|X| \geq 3$.  We will prove by induction on $|X|$ that $\G$ contains a triangle with at most one edge labeled 2.  Suppose $X=\{a,b,c\}$. If $a$ commutes with both $b$ and $c$, then $\AG \cong \Z \times A_{\{b,c\}}$ which contradicts our assumption that $\AG$ is infinite type.  
	
	Now suppose $X=\{x_1, \dots x_n\}, n>3$.  Since $\AG$ is infinite type it must contain an edge connecting some $x_i,x_j$ with $m_{ij} \neq 2$. If every triangle containing this edge has the other two edges labeled 2, then $x_i$ and $x_j$ commute with all other generators $x_k$.  Thus, $\AG \cong A_R \times A_T$ where $R= \{x_i,x_j\}$ and $T=X \backslash R$.  Since $A_R$ is spherical type, $A_T$ must be infinite type, so by induction, the subgraph spanned by $T$ contains a triangle with at most one edge labelled 2.
\end{proof}

\emph{From now on we will assume that our Artin groups $\AG$ are of infinite type.}
\medskip

A solution to the word problem for large type Artin groups is given in
\cite{HoltRees2012} and we will rely heavily on consequences of this algorithm. The algorithm works by successive manipulations of subwords of a
non-geodesic word until a free cancellation is possible. In \cite{BlascoCumplidoMW2022}, the algorithm is modified to apply to 3-free Artin groups by the addition of specific commutations between two generators.

We will say that two geodesics words are \emph{equivalent} if they represent the same element of $\AG$. We will use the following key takeaway from these papers to prove Criterion 1.

\begin{proposition}\textbf{Properties of
		rightward reducing sequences
		\cite[Proposition 3.3]{HoltRees2012}, \cite[Proposition 5.1] {BlascoCumplidoMW2022} \label{WordProbFact}}
	
	Let $\AG$ be a large type or 3-free Artin group, $w$ be a geodesic word over~$X$, and $x$ be a letter in $x\in X\cup X^{-1}$ such that $wx$ is not geodesic. Then there is a suffix $\zeta$ of $w$ and an algorithm called a rightward reducing sequence (RRS), such that when the RRS is applied to~$w$, $\zeta$ is transformed into an equivalent word $\zeta'$ that ends in $x^{-1}$. Furthermore, if $\zeta$ begins in a positive (resp.~negative) letter then $\zeta'$ ends in a positive (resp.~negative) letter. 
\end{proposition}

\subsection{Criterion 1: Preserving signed suffixes}

\begin{lemma}\label{Lem:LargeTypeSuffix}
	Let $\AG$ be an Artin group of large type or 3-free type. Then $\AG$ has preserved signed suffixes. 
\end{lemma}

\begin{proof}

	Suppose that $w$ is a geodesic word over $X$ in $\AG$, with longest positive suffix $u$, and suppose that $a$ is a positive generator from $X$.
	
	If $aw$ is geodesic, then $u$ is also a suffix of the geodesic word $aw$.
	
	If $aw$ is non-geodesic but not freely reduced, then the first letter of $w$ must be $a^{-1}$ and this cancellation must produce a geodesic representative of $aw$ that has $u$ as a suffix.
	
	So now suppose that $aw$ is freely reduced but not geodesic. Factor $aw$ as $aw_1w_2$ where $aw_1$ is the largest prefix of $aw$ that is geodesic. Denote by $x$ the first letter of $w_2$. 
	By \Cref{WordProbFact} there is an RRS applied to a suffix $\zeta$ of $aw_1$ that results in a new geodesic ending in $x^{-1}$. 
	This suffix $\zeta$ must start with the positive letter $a$, as otherwise the RRS could be applied to $w_1$ and $w$ would not be geodesic. So~$\zeta=aw_1$, and after performing the RRS algorithm, we get a word $\zeta'$ equivalent to~$aw_1$. By \Cref{WordProbFact}, $\zeta'$ ends in a positive letter, since $a$ was positive. We know this letter is $x^{-1}$, so we conclude that $x$ is a negative letter. This means that $u$, as a positive suffix of $aw$, must be a proper suffix of $w_2$.  
	
	Note also that $w$ is geodesic, so that any word representing the same group 
	element as $aw$ must have length at least $|w|_X-1$. For suppose to the contrary that there were a shorter word $\xi$ representing the same element as $aw$. Then $a^{-1}\xi$ is a word representing the same element as $w$, but with length smaller than $|w|_X$, contradicting $w$ being geodesic. This implies that applying a single RRS on $aw$ followed by the free cancellation between $x^{-1}$ and $x$, we obtain a geodesic representative of the same group element as $aw$. As explained in the previous paragraph, except for the cancellation of the first letter, the suffix $w_2$ has been unaffected by this algorithm, thus we have constructed a geodesic representative of $aw$ that ends in $u$. 
\end{proof}

\subsection{Criterion 2: Constructing alternating blocking sequences}

In this section we will explain how to construct an alternating blocking sequence for large type and 3-free Artin groups with at least 3 generators. Once this sequence is constructed we may apply \Cref{Thrm:Criteria} and \Cref{Cor:subgroups} to conclude that any infinite type Artin group $\AG$ that is either large type or 3-free, has infinite diameter monoid Cayley graph $\MCay$.

To prove Criterion 2 we rely on the following:

\begin{proposition} \textbf{Suffixes of equivalent geodesics \label{KeyConcepts2}}
	
	\begin{enumerate}
		\item \cite[Prop 4.5]{HoltRees2012} Suppose that $\AG$ is a large type Artin group and that $a,b\in X$ are distinct generators. 
		Given a geodesic word $wb^\pm a^\pm$, any equivalent geodesic word must end in either $a^\pm$ or $b^\pm$.  
		\item \label{EqGeo3free} Suppose that $\AG$ is a 3-free Artin group, and $a,b,c\in X$ are distinct generators, such that {$m_{a,b}>3$} and {$m_{b,c}>3$}
		
		\begin{enumerate}
			\item Given a geodesic word $wb^\pm a^\pm$, any equivalent geodesic must end in $a^\pm$, $b^\pm$ or a letter that commutes with $a$.
			\item Given a geodesic word $wb^\pm c^\pm b^\pm a^\pm$, any equivalent geodesic must end in $a^\pm$ or a letter that commutes with $a$. 
		\end{enumerate}
	\end{enumerate}
	
\end{proposition}

Part(2) above of the above proposition is new. We will prove Criterion 2 is satisfied for both large type and 3-free Artin groups assuming this result, and return to the proof of part (2) at the end of this section.

\subsubsection{Large type Artin groups}

\begin{lemma}\label{Lem:Criterion2Large} Let $\AG$ be an Artin group of large type. Then $\AG$ contains an alternating blocking sequence.
\end{lemma}
\begin{proof}
	Suppose that $\AG$ is a large type Artin group and that $a,b,c\in X$ are three distinct positive generators. We define an alternating blocking sequence by
	
	\[\alpha_1 = ab, \alpha_2=c^{-1}a^{-1}, \alpha_3=bc, \alpha_4=a^{-1}b^{-1}, \alpha_5=ca ,\alpha_6=b^{-1}c^{-1}\] and for $i>6$, $\alpha_i=\alpha_{i-6}$.
	
	This is a sequence of words in the standard generators such that $\alpha_i$ is a positive word when $i$ is odd and a negative word when $i$ is even. 
	It remains to verify that for any letter $x$ in $\alpha_i$ there exists some subword $u$ of $\alpha_{i-1}\alpha_i$ immediately preceding $x$ such that $(u,x)$ is a blocking pair. Without loss of generality we assume that $i=2\mod{6}$ as other arguments are symmetric. Thus it remains to show that both $(ab,c^{-1})$ and $(bc^{-1},a^{-1})$ are blocking pairs.

	Suppose to the contrary that, for some word $w$, $wab$ is geodesic, but $wabc^{-1}$ is not. Then by
	\Cref{WordProbFact}, the word $wab$ admits a rightward reducing sequence that transforms it to a word $w'$ such that the final letter of $w'$ is $c$.
	However, since $w'$ is a geodesic representative of $wab$, this contradicts part (1) of \Cref{KeyConcepts2}, which requires that every geodesic representative of $wab$ must end with a letter in the set~$\{a,b,a^{-1},b^{-1}\}$. Hence $(ab,c^{-1})$ is a blocking pair. Notice in the above argument we used only the fact that $a,b,c$ were distinct generators in a large type Artin group, so an identical argument can be used to show that $(bc^{-1},a^{-1})$ is a blocking pair. 
	This shows that the given sequence $\{\alpha_i\}$ is an alternating blocking sequence.
\end{proof}

\subsubsection{3-free Artin groups}

\begin{lemma}\label{Lem:Criterion2} Let $\AG$ be a 3-free Artin group of infinite type such that no~$m_{ij}=\infty$. Then $\AG$ contains an alternating blocking sequence.
\end{lemma}
\begin{proof}
	By \Cref{Lem:3-free}, $\Gamma$ contains three generators $a,b,c$ such that at most one of the relations is a commutation. We have also already covered large type Artin groups, so we can assume that 
	\[m_{a,b}>3, \quad m_{b,c}>3, \quad m_{a,c}=2.\]
	
	\noindent We define an alternating sequence of words by
	\[
	\alpha_i=babc \text{ if $i$ is odd,} \quad
	\alpha_i=b^{-1}a^{-1}b^{-1}c^{-1} \text{ if $i$ is even.}\]
	This sequence alternates between positive and negative words in the standard generators, and so it remains to show that for any letter $x$ in $\alpha_i$ there exists some subword $u$ of $\alpha_{i-1}\alpha_i$ immediately preceding $x$ such that $(u,x)$ is a blocking pair. 
	Assuming without loss of generality that $i$ is even, it follows as in the proof of \Cref{Lem:Criterion2Large} that $(babc,b^{-1})$, $(cb^{-1},a^{-1})$, $(bcb^{-1}a^{-1},b^{-1})$, and $(a^{-1}b^{-1},c^{-1})$ are all blocking pairs, using part (2) of \Cref{KeyConcepts2} in place of part (1). 
	Note that in some cases part (2) of \Cref{KeyConcepts2} is applied to the three generators ordered as $c,b,a$ rather than as $a,b,c$, which is valid because of
	the symmetry in the triple $\{m_{a,b},m_{a,c}, m_{b,c}\}$.
\end{proof}

We conclude this section with the proof of part (2) of \Cref{KeyConcepts2}, for which we require more technical elements of the solution to the word problem as given in \cite{BlascoCumplidoMW2022}. Note that we essentially follow the notation from \cite{BlascoCumplidoMW2022}, except that we have used subwords $\xi_i$ here where subwords $w_i$ are used in \cite{BlascoCumplidoMW2022}, in order to avoid confusion with our previously defined sequence $w_n$.

Recall that given a word $u$ we can write 
$u=x_1^{k_1}x_2^{k_2}\cdots x_n^{k_n}$ where we assume the word is freely reduced and $x_i\neq x_{i+1}$. We call $n$ the syllable length of this word.

\begin{remark}
	To prove part (2) of \Cref{KeyConcepts2}, we will use the following key facts from \cite{BlascoCumplidoMW2022}. These facts follow directly from the way the algorithm for the solution to the word problem is defined in that paper. 
	Suppose that a word $w$ can be changed to a new word $w'$ via a rightward reducing sequence. 
	Then $w$ has a suffix $\xi_1\xi_2\cdots \xi_{k+1}$ satisfying the following properties. 
	
	\begin{itemize}
		\item[F1] For $1\leq i\leq k$,  the first $i$ steps of the algorithm involve a modification to $\xi_1\cdots \xi_i$ and leave the suffix $\xi_{i+1}\cdots \xi_{k+1}$ unaffected. 
		\item[F2] The final $(k+1)^{st}$ step in the algorithm involves only commutations of individual letters. Thus if $w'$ ends in a letter $x$, then either $x$ commutes with all the letters in $\xi_{k+1}$, or $\xi_{k+1}$ is the empty word, in which case the algorithm only involves $k$ steps. 
		
		\item[F3]The 3-free condition, together with \cite[Lemma 3.11]{BlascoCumplidoMW2022} implies that the syllable length of $\xi_i$ must be at least 2 for all $i\leq k$. 
		\item[F4] If $\xi_i$ has $x^\pm y^\pm$ as a suffix and $m_{x,y}>3$, then after the application of the $i^{\rm th}$ step in the algorithm, the word created has suffix $\mu \xi_{i+1}\cdots \xi_{k+1}$, where $\mu$ is a word in $x,y$ of syllable length 2.
		
		\item[F5] The subword $\xi_i$ cannot end in $x^\pm y^\pm z^\pm$, where $m_{y,z}>3$ and $m_{x,y}>3$.
	\end{itemize}
\end{remark}
\begin{proof}[Proof of \Cref{KeyConcepts2} (2)]
	
	To prove part (a), suppose that $u$ is a geodesic word ending in the letter $x$ and representing the same group element as $wb^\pm a^\pm$. Then $wb^\pm a^\pm x^{-1}$ is not geodesic and it follows from \Cref{WordProbFact} that either $wb^\pm a^\pm x^{-1}$ is not freely reduced, in which case, $x=a^\pm$ and we are done, or the word  $wb^\pm a^\pm$ can be changed via a rightward reducing sequence to a new word $u'$ ending in $x$.  
	
	Our goal is to show that one of the cases $x=a^\pm$, $x=b^\pm$ or $m_{x,a}=2$ holds. We consider a factorization of $wb^\pm a^\pm$ with suffix $\xi_1\xi_2\cdots \xi_{k+1}$. 
	If $\xi_{k+1}$ is non-empty, then it ends in $a^{\pm}$, and then by (F2) $x$ must commute with $a$. Otherwise, $\xi_{k+1}$ is empty. Then by (F3) $\xi_{k}$ has $b^\pm a^\pm$ as a suffix and, by (F4), after the $k^{th}$ and final step in the algorithm, the result is a word that ends in $a^\pm$ or $b^\pm$.
	
	Now we prove part (b). As for part (a), we consider a word $u$ ending in $x$ and equivalent to $wb^\pm c^\pm b^\pm a^\pm$. If $wb^\pm c^\pm b^\pm a^\pm x^{-1}$ is not geodesic, then either $wb^\pm c^\pm b^\pm a^\pm x^{-1}$ is not freely reduced and $x=a^\pm$ or we can use a rightward reducing sequence to obtain a new word $u'$ ending in $x$ and equivalent to $wb^\pm c^\pm b^\pm a^\pm$ by \Cref{WordProbFact}. We will show that the suffix $b^\pm c^\pm b^\pm a^\pm$ implies there is no such sequence. 
	
	We have already shown in the proof of part (a) that if $\xi_{k+1}$ is non-empty then $x$ must commute with $a$, and if $\xi_{k+1}$ is empty then $b^{\pm}a^{\pm}$ must be a suffix of $\xi_k$. We need now to consider further the case where $\xi_{k+1}$ is empty, and consider the two cases where $c^\pm$ is a part of $\xi_k$ and where $c^\pm $ is a part of $\xi_{k-1}$. In the first case, $c^\pm b^\pm a^\pm$ is a suffix of the word $\xi_k$, but then, since $m_{b,c}>3$ and $m_{a,b}>3$, by (F5) the algorithm cannot proceed to achieve our word ending in~$x$, a contradiction.
	
	If, on the other hand, $c^\pm$ is part of $\xi_{k-1}$, then $\xi_{k-1}$ must have $b^\pm c^\pm$ as a suffix. Then (F4) implies that after the $(k-1)^{st}$ step in the algorithm the word still ends in  $\mu b^\pm a^\pm$, where $\mu$ is a word in $b,c$ of syllable length $2$, that is, $\mu b^\pm a^\pm$ is either of the form $c^ib^ja^{\pm}$ or $b^jc^ib^{\pm}a^{\pm}$.  In either case, (F5) implies the algorithm cannot proceed, which is again a contradiction.
\end{proof}

\section{Conclusions} \label{Section:discussion} 
Combining  \Cref{Thrm:Criteria} with Lemmas~\ref{Lem:LargeTypeSuffix}, \ref{Lem:Criterion2Large} and \ref{Lem:Criterion2}, we conclude that large type and 3-free Artin groups (of infinite type) have infinite diameter monoid Cayley graphs.  Together with Corollary~\ref{Cor:subgroups}, we thus obtain our main theorem.

\begin{theorem}\label{Thm:Main}  Let $\MCay$ denote the Cayley graph of $\AG$ with respect to the generating set $M:=\AG^+$.  Then $\MCay$ has infinite diameter providing one of the following holds.
	\begin{enumerate}
		\item $\G$ contains a pair of vertices not connected by an edge (or equivalently, there exists $i,j$ with $m_{i,j}=\infty$).
		\item $\G$ contains a triangle with all labels $\geq 3$.
		\item $\G$ contains a triangle with no label equal to $3$ and at most one label equal to $2$. 
	\end{enumerate}
\end{theorem}

Recall that Conjecture \ref{Conj:InfDiam} states that $\MCay$ has infinite diameter for all infinte type Artin groups, To complete the proof of this conjecture, it remains to consider the case in which all 3-generator special subgroup of $\AG$ are either spherical type, or correspond to triangles labelled $(2,3,n), n\geq 6$.  We believe that the latter case can be solved by arguments similar to those used in the large type and 3-free cases.

\bibliographystyle{alpha}
\bibliography{references.bib}		

\end{document}